\documentclass{amsart}
\usepackage[bookmarksnumbered, colorlinks, plainpages]{hyperref}
\usepackage{amsmath,amsfonts,amssymb}
\usepackage{url}
\usepackage{mathrsfs}
\usepackage{graphicx}
\usepackage{caption}
\usepackage{subfig}
\usepackage{alltt}
\usepackage{epigraph}
\usepackage{cite}
\usepackage[shortlabels]{enumitem}
\usepackage{mathtools}

\theoremstyle{plain}
\newtheorem{theorem}{Theorem}[section]
\newtheorem{lemma}[theorem]{Lemma}
\newtheorem{proposition}[theorem]{Proposition}
\newtheorem{corollary}[theorem]{Corollary}

\theoremstyle{definition}

\newtheorem{example}[theorem]{Example}

\theoremstyle{remark}
\newtheorem{remark}[theorem]{Remark}

\newcommand{\norm}[1]{\|#1\|}
\newcommand{\scal}[1]{\langle#1\rangle}
\DeclareMathOperator{\ran}{ran}
\DeclareMathOperator{\rank}{rank}
\DeclareMathOperator{\spc}{Sp}
\newcommand{\R}{\mathbb{R}}
\newcommand{\T}{\mathbb{T}}
\begin{document}

\title[Operator equalities and Characterizations of  Orthogonality]{Operator equalities and Characterizations of Orthogonality in Pre-Hilbert $C^*$-Modules}

\author[R. Eskandari, M.S. Moslehian, D. Popovici]{Rasoul Eskandari$^1$, M. S. Moslehian$^2$, \MakeLowercase{and} Dan Popovici$^3$}
\address{$^1$Department of Mathematics, Faculty of Science, Farhangian University, Tehran, Iran.}
\email{eskandarirasoul@yahoo.com}

\address{$^2$Department of Pure Mathematics, Center of Excellence in
Analysis on Algebraic Structures (CEAAS), Ferdowsi University of Mashhad, P. O. Box 1159, Mashhad 91775, Iran.}
\email{moslehian@um.ac.ir}

\address{$^3$Department of Mathematics and Computer Science, West University of Timisoara, RO-300223 Timisoara, Vasile Parvan Blvd no. 4, Romania.}
\email{dan.popovici@e-uvt.ro}



\begin{abstract}
In the first part of the paper, we use states on $C^*$-algebras in order to establish some equivalent statements to equality in the triangle inequality, as well as to the parallelogram identity for elements of a pre-Hilbert $C^*$-module. We also characterize the equality case in the triangle inequality for adjointable operators on a Hilbert $C^*$-module. Then we give certain necessary and sufficient conditions to the Pythagoras identity for two vectors in a pre-Hilbert $C^*$-module under the assumption that their inner product has negative real part. We introduce the concept of Pythagoras orthogonality and discuss its properties. We describe this notion for Hilbert space operators in terms of the parallelogram law and some limit conditions. We present several examples in order to illustrate the relationship between the Birkhoff--James, Roberts, and Pythagoras orthogonalities, and the usual orthogonality in the framework of Hilbert $C^*$-modules.
\end{abstract}

\maketitle

\section{Introduction}

Let $S(\mathcal{A})$ be the set of all states of a given $C^*$-algebra $\mathcal{A}$. The \textit{numerical range} of an element $a\in\mathcal{A}$ is defined by
\begin{equation*}
V(a)=\{\varphi(a):\varphi\in S(\mathcal{A})\}.
\end{equation*}
If $a$ is a normal element of $\mathcal{A}$ then there exists a state $\varphi$ on $\mathcal{A}$ such that $|\varphi(a)|=\|a\|$ (cf. \cite[Theorem 3.3.6]{Mur}). The set
\begin{equation*}
S_a(\mathcal{A})=\{\varphi\in S(\mathcal{A}):|\varphi(a)|=\|a\|\}
\end{equation*}
is nonempty and closed. This set is also convex if $a$ is positive.

A (right) \textit{pre-Hilbert $C^*$- module} $\mathscr{E}$ over a $C^*$-algebra $\mathcal{A}$ is a (complex) linear space which is also a right $\mathcal{A}$-module, having a compatible structure (i.e., $\lambda(xa)=(\lambda x)a=x(\lambda a),\ \lambda\in\mathbb{C},a\in\mathcal{A},x\in\mathscr{E}$), equipped with an $\mathcal{A}$-valued inner product on $\mathscr{E}$, i.e., a sesquilinear map $\langle \cdot,\cdot\rangle: \mathscr{E}\times \mathscr{E}\to \mathcal{A}$ with the properties:
\begin{enumerate}[$(a)$]
\item $\scal{x,x}\ge 0,\ x\in\mathscr{E}$; $\scal{x,x}=0$ if and only if $x=0$.
\item $\scal{x,y}^*=\scal{y,x},\ x,y\in\mathscr{E}$.
\item $\scal{x,ya}=\scal{x,y}a,\ x,y\in\mathscr{E},\ a\in\mathcal{A}$.
\end{enumerate}
The formula
\begin{equation*}
\mathscr{E}\ni x\mapsto\|x\|:=\||x|\|_\mathcal{A}\in\R_+
\end{equation*}
defines a norm on $\mathscr{E}$ (for $x\in\mathscr{E}$, the notation $|x|:=\scal{x,x}^{1/2}$ will be used in the subsequent part of the paper).
A pre-Hilbert $\mathcal{A}$-module which is complete with respect to this norm is called a \textit{Hilbert $C^*$-module over $\mathcal{A}$}, or a \textit{Hilbert $\mathcal{A}$-module}. Every $C^*$-algebra $\mathcal{A}$ can be regarded as a Hilbert module over itself, the inner product being defined as $\scal{a,b}:=a^*b,\ a,b\in\mathcal{A}$.

Suppose that $\mathscr{E}$ and $\mathscr{F}$ are Hilbert $C^*$-modules. Let $\mathcal{L}(\mathscr{E}, \mathscr{F})$ be the set of all maps $T :\mathscr{E}\to \mathscr{F}$ for which there is an application $T^*:\mathscr{F}\to\mathscr{E}$ such that
\begin{equation}\label{eqdefadjoint}
\langle Tx, y\rangle =\langle x, T^*y\rangle,\quad x\in\mathscr{E}, y\in\mathscr{F}.
\end{equation}
An operator $T\in\mathcal{L}(\mathscr{E},\mathscr{F})$, called \textit{adjointable}, is $\mathcal{A}$-linear and bounded, while $T^*$ (the \textit{adjoint} of $T$) is uniquely determined by \eqref{eqdefadjoint}. The map $T\mapsto T^*$ has the properties of an isometric involution. Moreover, $\mathcal{L}(\mathscr{E}):=\mathcal{L}(\mathscr{E},\mathscr{E})$ is a $C^*$-algebra.

Thus, Hilbert $C^*$-modules are generalization of Hilbert spaces by allowing inner products to take values in a $C^{*}$-algebra rather than in the field of complex numbers. Unfortunately, certain basic properties of Hilbert spaces are not valid in general Hilbert $C^*$-modules. For example, it is not true that any bounded linear operator on a Hilbert $C^*$-module is adjointable or any closed submodule is orthogonally complemented. Therefore, not only any investigation in the context of Hilbert $C^*$-modules is non-trivial, but also it is an interesting question to ask under which conditions the results analogous to those for Hilbert spaces can still remain true for Hilbert $C^*$-modules.

It is known that the equality $\|x + y\|=\|x\| +\|y\|$ holds in a Hilbert space $\mathscr{H}$ if and only if $x$ and $y$ are linearly dependent by positive scalars. Being a starting point in our discussion on Pythagoras identities, one of our goals is to investigate the validity of this equality in the setting of Hilbert $C^*$- modules. Maybe the first result in this direction is a characterization of Aramba\v{s}i\'c and Raji\'c \cite{AR2a} which shows that, for two elements $x$ and $y$ in a pre-Hilbert $\mathcal{A}$-module $\mathscr{E}$, $\|x + y\|=\|x\| +\|y\|$ if and only if $\|x\|\|y\|\in V(\scal{x,y})$. The particular situations of Hilbert space operators or of elements in a $C^*$-algebra have been emphasized earlier by Barraa and Boumazgour \cite{BB2}, respectively by Nakamoto and Takahasi \cite{NT}. We show, among others, that the following statements are equivalent: $\||x|^2+|y|^2\|=\|x\|^2+\|y\|^2$; $\||x||y|\|=\|x\|\|y\|$; $S_{|x|^2}(\mathcal{A})\cap S_{|y|^2}(\mathcal{A})\ne\emptyset$; $\|x\|^2\|y\|^2\in V(|x|^2|y|^2)$; $\|x\|^2+\|y\|^2\in V(|x|^2+|y|^2)$. We also discuss the ``triangle equality'' for two adjointable operators $s$ and $t$ on $\mathscr{E}$. By contrast with the earlier approach, our result relies on the states of $\mathcal{A}$ and not on the states of the $C^*$-algebra $L(\mathscr{E})$. More precisely, we prove that $\|s+t\|=\|s\|+\|t\|$ if and only if there exist sequences $(\varphi_n)_{n\ge 0}$ (of states on $\mathcal{A}$) and $(x_n)_{n\ge 0}$ (of elements in $\mathscr{E}$) such that $\varphi_n(|x_n|^2)=1,\ n\ge 0$ and $\varphi_n(\scal{sx_n,tx_n})\xrightarrow{n\to\infty}\|s\|\|t\|$.

A norm $\|\cdot\|$ on a vector space $\mathscr{X}$ is induced by a scalar product if and only if the parallelogram identity $\|x+y\|^2+\|x-y\|^2=2(\|x\|^2+\|y\|^2)$ holds for every $x,y\in\mathscr{X}$. This parallelogram identity is not valid in the general framework of Hilbert $C^*$-modules. In our attempt to characterize this notion using the language of states we show that any two of the following statements imply the third one: $x$ and $y$ verify the parallelogram identity; $S_{|x|^2}(\mathcal{A})\cap S_{|y|^2}(\mathcal{A})\ne\emptyset$; $S_{|x+y|^2}(\mathcal{A})\cap S_{|x-y|^2}(\mathcal{A})\ne\emptyset$.

Our next aim was to characterize an equality of the form $\|x+y\|^2=\|x\|^2+\|y\|^2$ (\textit{Pythagoras identity}). This identity has been studied by many authors, in various contexts, starting with James \cite{Jam1}. We prove, under the assumption that the inner product $\scal{x,y}$ has negative real part, that the following statements are equivalent: $\|x+y\|^2=\|x\|^2+\|y\|^2$; $\|x\|^2+\|y\|^2\in V(|x+y|^2)$; there exists $\varphi\in S_{|x|^2}(\mathcal{A})\cap S_{|y|^2}(\mathcal{A})$ such that $\varphi(\Re(\scal{x,y}))=0$.

In the general context of (complex) normed linear spaces $\mathscr{X}$, there were several attempts to extend the notion of orthogonality for two vectors $x$ and $y$. More exactly, $x$ and $y$ are \textit{orthogonal in the Roberts sense} (in notation, $x\perp_Ry$; cf. \cite[p. 56]{Rob}) if $\norm{x+\lambda y}=\norm{x-\lambda y},\ \lambda\in\mathbb{C}$. The concept of \textit{Birkhoff-James orthogonality} (in notation, $x\perp_By$), has been suggested by G. Birkhoff \cite{Bir} and R.C. James \cite{Jam} as $\norm{x+\lambda y}\ge\norm{x},\ \lambda\in\mathbb{C}$. In the framework of pre-Hilbert $C^*$-modules these notions have been studied, for example, in \cite{AR1,AR4,BG,ZM}.

The main part of this paper is devoted to the study of another concept of orthogonality, namely the \textit{Pythagoras orthogonality}. A vector $x$ is said to be \textit{orthogonal in the Pythagoras sense} to a vector $y$ (in notation, $x\perp_P y$) if
\begin{equation*}
\norm{x+\lambda y}^2=\norm{x}^2+|\lambda|^2\norm{y}^2,\quad\lambda\in\mathbb{C}.
\end{equation*}
If $x\perp_P y$ then, clearly, $x$ and $y$ \textit{satisfy the parallelogram law}, that is
\begin{equation*}
\norm{x+\lambda y}^2+\norm{x-\lambda y}^2=2(\norm{x}^2+|\lambda|^2\norm{y}^2),\quad\lambda\in\mathbb{C}.
\end{equation*}
We start by presenting the main properties of Pythagoras orthogonality and discuss its relationship with the parallelogram law, Roberts orthogonality, Birkhoff--James orthogonality and inner product orthogonality. Pythagoras orthogonality implies both the parallelogram law and Birkhoff--James orthogonality. We show that, for two elements $x$ and $y$ in $\mathscr{E}$ (a pre-Hilbert module over a unital $C^*$-algebra) such that $|y|^2$ is a positive multiple of the identity, the converse is also true. We finally characterize the Pythagoras orthogonality for two operators $A$ and $B$ in $\mathcal{L}(\mathscr{H})$ (regarded as a Hilbert module over itself) as follows. Under the assumptions that $\rank(A+\alpha_1 B)>1$ and $\Re(\alpha_2A^*B)\ge 0$ for certain $\alpha_1,\alpha_2\in\mathbb{C},\ \alpha_2\ne 0$, $A$ and $B$ are orthogonal in the Pythagoras sense if and only if $A$ and $B$ verify the parallelogram law and there exists a sequence $(\xi_n)_{n\ge 0}$ of unit vectors in $\mathscr{H}$ such that $\|A\xi_n\|\xrightarrow{n\to\infty}\|A\|,\ \|B\xi_n\|\xrightarrow{n\to\infty}\|B\|$ and $\scal{A\xi_n,B\xi_n}\xrightarrow{n\to\infty}0$ if and only if $A$ and $B$ verify the parallelogram law and there exists a sequence $(\xi_n)_{n\ge 0}$ of unit vectors in $\mathscr{H}$ such that $\|(A+\lambda B)\xi_n\|^2\xrightarrow{n\to\infty}\|A\|^2+|\lambda|^2\|B\|^2$ for every $\lambda\in\mathbb{C}$. Several examples are given for illustrative purposes.

\section{``Triangle Equalities''}

We start our work by the observation that the equality case in the triangle inequality for two elements $x$ and $y$ in a normed linear space $\mathscr{X}$ is preserved for their positive multiples $\alpha x\ (\alpha\ge 0)$ and $\beta y\ (\beta\ge 0)$.
\begin{lemma}[\protect{\cite[Lemma 2.1]{Abr}}]\label{lem2}
Let $x$ and $y$ be two vectors in a normed linear space $\mathscr{X}$ such that
$\|x+y\|=\|x\|+\|y\|$. Then $\|\alpha x+\beta y\| =\alpha\|x\|+\beta\|y\|$ for every $\alpha,\beta\ge 0$.
\end{lemma}

An equality of the form $\|\alpha x+\beta y\|=|\alpha|\|x\|+|\beta|\|y\|$ ($\alpha,\beta\in\mathbb{C},\ \alpha,\beta\ne 0$) can be reformulated for scalars $\alpha,\beta$ belonging to the unit circle $\T$. More precisely, the following holds.
\begin{proposition}
Let $x$ and $y$ be two vectors of a normed linear space $\mathscr{X}$. The following statements are equivalent:
\begin{enumerate}[$(i)$]
\item $\|\alpha x+\beta y\|=|\alpha|\|x\|+|\beta|\|y\|$ for some nonzero scalars $\alpha,\beta\in\mathbb{C}$.
\item $\|\alpha x+\beta y\|=\| x\|+\|y\|$ for some $\alpha,\beta\in\T$.
\end{enumerate}
\end{proposition}

\begin{proof}
We only have to prove the implication $(i)\Rightarrow(ii)$, the other one is obvious. Let $x'=\alpha x$ and $y'=\beta y$. Then $(i)$ takes the form $\|x'+y'\|=\|x'\|+\|y'\|$ so, by Lemma \ref{lem2}, $\|\frac{1}{|\alpha|}x'+\frac{1}{|\beta|}y'\|=\frac{1}{|\alpha|}\|x'\|+\frac{1}{|\beta|}\|y'\|$. In other words,
\begin{equation*}
\biggl\|\frac{\alpha}{|\alpha|}x+\frac{\beta}{|\beta|}y\biggr\|=\|x\|+\|y\|.
\end{equation*}
\end{proof}

The following result, characterizing the equality case in the triangle inequality for two elements of a pre-Hilbert $\mathcal{A}$-module has been formulated in \cite[Proposition 3]{Pop} using a representation of $\mathcal{A}$ on a Hilbert space. It will be presented here using the terminology of states. We would also like to mention that the equivalence $(i)\Leftrightarrow(iii)$ has been obtained in \cite[Theorem 2.1]{AR2a}.

\begin{proposition}[\protect{\cite[Proposition 3]{Pop}}]
Let $x,y$ be two elements in a pre-Hilbert module over a $C^*$-algebra $\mathcal{A}$. The following statements are equivalent:
\begin{enumerate}[$(i)$]
\item $\|x+y\|=\|x\|+\|y\|$.
\item $(\|x\|+\|y\|)^2\in V(|x+y|^2)$.
\item $\|x\|\|y\|\in V(\scal{x,y})$.
\end{enumerate}
If $\varphi$ is a given state on $\mathcal{A}$, then $\varphi(|x+y|^2)=(\|x\|+\|y\|)^2$ if and only if $\varphi(\scal{x,y})=\|x\|\|y\|$. In this case, $\varphi\in S_{|x|^2}(\mathcal{A})\cap S_{|y|^2}(\mathcal{A})$ and $\varphi(\scal{x,y}^*\scal{x,y})=\|x\|^2\|y\|^2$.
\end{proposition}

We now describe a triangle ``equality'' in the context of pre-Hilbert $C^*$-modules.

\begin{proposition}\label{c5}
Let $x$ and $y$ be two elements of a pre-Hilbert $\mathcal{A}$-module. The following statements are equivalent:
\begin{enumerate}[$(i)$]
\item $\||x|^2+|y|^2\|=\|x\|^2+\|y\|^2$.
\item $\||x||y|\|=\|x\|\|y\|$.
\item $S_{|x|^2}(\mathcal{A})\cap S_{|y|^2}(\mathcal{A})\ne\emptyset$.
\item $\|x\|^2\|y\|^2\in V(|x|^2|y|^2)$.
\item $\|x\|^2+\|y\|^2\in V(|x|^2+|y|^2)$.
\end{enumerate}
\end{proposition}
\begin{proof}
We only prove $(v)\Rightarrow(i)$, the rest can be concluded from \cite[Proposition 3.3]{Kit},  \cite[Theorem 2.1]{AR2a}, and \cite[Theorem 1]{NT}). To this end, it is enough to show that  if $a$ and $b$ are two positive elements of a $C^*$-algebra $\mathcal{A}$, and  $\|a\|+\|b\|\in V(a+b)$, then $\|a+b\|=\|a\|+\|b\|$:

Let $\varphi\in S(\mathcal{A})$ be such that $\varphi(a+b)=\|a\|+\|b\|$. Then
\begin{equation*}
\|a\|+\|b\|=\varphi(a+b)\le\|a+b\|.
\end{equation*}
We deduce immediately that $\|a+b\|=\|a\|+\|b\|$, as required.
\end{proof}

In particular, for any $C^*$-algebra $\mathcal{A}$ (regarded as a Hilbert $\mathcal{A}$-module), the following holds.

\begin{corollary}\label{c6}
Let $a$ and $b$ be two elements of $\mathcal{A}$. The following statements are equivalent:
\begin{enumerate}[$(i)$]
\item $\|a^*a+b^*b\|=\|a\|^2+\|b\|^2$.
\item $\|ab^*\|=\|a\|\|b\|$.
\end{enumerate}
\end{corollary}

\begin{proof}
The condition $(i)$ is equivalent with $\|a^*ab^*b\|=\|a\|^2\|b\|^2$ (see, \cite[Proposition 3.3]{Kit}).

The implication $(i)\Rightarrow(ii)$ follows by the inequalities:
\begin{equation*}
\|a\|^2\|b\|^2=\|a^*ab^*b\|\le\|a^*\|\|ab^*\|\|b\|\le\|a\|^2\|b\|^2.
\end{equation*}

Conversely, if $(ii)$ holds true, then
\begin{equation*}
\|a\|^4\|b\|^4=\|ab^*\|^4=\|ab^*ba^*ab^*ba^*\|\le\|a\|\|b\|^2\|a^*ab^*b\|\|a\|\le\|a\|^4\|b\|^4.
\end{equation*}
Consequently, $\|a^*ab^*b\|=\|a\|^2\|b\|^2$, which is equivalent with $(i)$, by Proposition \ref{c5} .
\end{proof}

\begin{corollary}
Let $x$ and $y$ be two elements of a pre-Hilbert $\mathcal{A}$-module. Then any two of the following statements imply the third one:
\begin{enumerate}[$(i)$]
\item $x$ and $y$ verify the parallelogram law:
\begin{equation*}
\|x+y\|^2+\|x-y\|^2=2(\|x\|^2+\|y\|^2).
\end{equation*}
\item $S_{|x|^2}(\mathcal{A})\cap S_{|y|^2}(\mathcal{A})\ne\emptyset$.
\item $S_{|x+y|^2}(\mathcal{A})\cap S_{|x-y|^2}(\mathcal{A})\ne\emptyset$.
\end{enumerate}
\end{corollary}

\begin{proof}
By Proposition \ref{c5} condition $(ii)$ is equivalent to $\||x|^2+|y|^2\|=\|x\|^2+\|y\|^2$. Similarly, condition $(iii)$ can be replaced by:
\begin{equation*}
2\||x|^2+|y|^2\|=\||x+y|^2+|x-y|^2\|=\|x+y\|^2+\|x-y\|^2.
\end{equation*}
The conclusion then follows easily.
\end{proof}

Our next aim is to describe the equality case in the triangle inequality for two adjointable operators in a Hilbert $C^*$-module.
\begin{theorem}
Let $\mathscr{E}$ be a Hilbert module over the $C^*$-algebra $\mathcal{A}$, and let $s,t\in\mathcal{L}(\mathscr{E})$. The following conditions are equivalent:
\begin{enumerate}[$(i)$]
\item $\|s+t\|=\|s\|+\|t\|$.
\item There exist sequences $(\varphi_n)_{n\ge 0}$ (of states on $\mathcal{A}$) and $(x_n)_{n\ge 0}$ (of elements in $\mathscr{E}$) such that $\varphi_n(|x_n|^2)=1,\ n\ge 0$ and
\begin{equation*}
\varphi_n(\scal{sx_n,tx_n})\xrightarrow{n\to\infty}\|s\|\|t\|.
\end{equation*}
\item There exist sequences $(\varphi_n)_{n\ge 0}$ (of states on $\mathcal{A}$) and $(x_n)_{n\ge 0}$ (of elements in $\mathscr{E}$) such that $\varphi_n(|x_n|^2)\le 1,\ n\ge 0$ and
\begin{equation*}
\varphi_n(\scal{sx_n,tx_n})\xrightarrow{n\to\infty}\|s\|\|t\|.
\end{equation*}
\end{enumerate}
\end{theorem}

\begin{proof}
It has been indicated in \cite[p. 37]{Lan} that, for any given state $\varphi$ of $\mathcal{A}$ and $x\in\mathscr{E}$ with $\varphi(|x|^2)=1$, the map $s\mapsto\varphi(\scal{x,sx})$ is a state of $\mathcal{L}(\mathscr{E})$. In addition, for any adjointable operator $s$ on $\mathscr{E}$,
\begin{equation}\label{eqlan}
\|s\|^2=\sup_{\phi(|x|^2)=1}\varphi(|sx|^2).
\end{equation}

$(i)\Rightarrow(ii)$. Let us consider, in view of \eqref{eqlan}, a sequence $(\varphi_n)_{n\ge 0}$ of states on $\mathcal{A}$ and a sequence $(x_n)_{n\ge 0}$ of elements in $\mathscr{E}$ such that $\varphi_n(|x_n|^2)=1,\ n\ge 0$ and
\begin{equation*}
\varphi_n(|(s+t)x_n|^2)\xrightarrow{n\to\infty}\|s+t\|^2.
\end{equation*}
We note that, for any $n\ge 0$,
\begin{equation*}
\begin{split}
\varphi_n(|(s+t)x_n|^2)& =\varphi_n(|sx_n|^2)+\varphi_n(|tx_n|^2)+\varphi_n(\scal{sx_n,tx_n})+\varphi_n(\scal{tx_n,sx_n})\\
& \le\|s\|^2+\|t\|^2+\|s^*t\|+\|t^*s\|\\
& \le(\|s\|+\|t\|)^2.
\end{split}
\end{equation*}
We pass to limit (as $n\to\infty$) to deduce, by $(i)$, that
\begin{equation*}
\varphi_n(\scal{sx_n,tx_n})\xrightarrow{n\to\infty}\|s\|\|t\|,
\end{equation*}
which proves $(ii)$.

The implication $(ii)\Rightarrow(iii)$ is obvious.

$(iii)\Rightarrow(i)$. Let $(\varphi_n)_{n\ge 0}$, and let $(x_n)_{n\ge 0}$ be sequences as in $(iii)$. By passing to limit (as $n\to\infty$) in the inequalities
\begin{equation*}
|\varphi_n(\scal{sx_n,tx_n})|\le\varphi_n(|sx_n|^2)^{1/2}\varphi_n(|tx_n|^2)^{1/2}\le\|s\|\|t\|\varphi_n(|x_n|^2)\le\|s\|\|t\|,\ n\ge 0,
\end{equation*}
we obtain that
\begin{equation*}
\varphi_n(|sx_n|^2)\xrightarrow{n\to\infty}\|s\|^2\text{ and }\varphi_n(|tx_n|^2)\xrightarrow{n\to\infty}\|t\|^2.
\end{equation*}
Hence,
\begin{equation*}
\varphi_n(|(s+t)x_n|^2)=\varphi_n(|sx_n|^2)+\varphi_n(|tx_n|^2)+2\Re\varphi_n(\scal{sx_n,tx_n})\xrightarrow{n\to\infty}(\|s\|+\|t\|)^2.
\end{equation*}
Letting again $n\to\infty$ in the inequalities $\varphi_n(|(s+t)x_n|^2)\le\|s+t\|^2\le(\|s\|+\|t\|)^2,\ n\ge 0$ we finally get the triangle ``equality'' in $(i)$.
\end{proof}

\begin{corollary}
Let $\mathcal{A}$ be a $C^*$-algebra, and let $a,b\in\mathcal{A}$. The following conditions are equivalent:
\begin{enumerate}[$(i)$]
\item $\|a+b\|=\|a\|+\|b\|$.
\item There exist sequences $(\varphi_n)_{n\ge 0}$ (of states on $\mathcal{A}$) and $(c_n)_{n\ge 0}$ (of elements in $\mathcal{A}$) such that $\varphi_n(c_n^*c_n)=1,\ n\ge 0$ and
\begin{equation*}
\varphi_n(c_n^*a^*bc_n)\xrightarrow{n\to\infty}\|a\|\|b\|.
\end{equation*}
\item There exist sequences $(\varphi_n)_{n\ge 0}$ (of states on $\mathcal{A}$) and $(c_n)_{n\ge 0}$ (of elements in $\mathcal{A}$) such that $\varphi_n(c_n^*c_n)\le 1,\ n\ge 0$ (or, in particular, $\|c_n\|\le 1,\ n\ge 0$) and
\begin{equation*}
\varphi_n(c_n^*a^*bc_n)\xrightarrow{n\to\infty}\|a\|\|b\|.
\end{equation*}
\end{enumerate}
\end{corollary}

\section{Pythagoras Identities}

We characterize the Pythagoras identity for two vectors in a pre-Hilbert $C^*$-modules under the assumption that their inner product has negative real part.
\begin{proposition}\label{p1}
Let $x$ and $y$ be two elements in a pre-Hilbert $C^*$-module $\mathscr{E}$ such that $\Re(\scal{x,y})\le 0$. The following conditions are equivalent:
\begin{itemize}
\item[$(i)$] $\|x+y\|^2=\|x\|^2+\|y\|^2.$
\item[$(ii)$] $\||x|^2+2\Re(\scal{x,y})+|y|^2\|=\||x|^2+|y|^2\|$ and $\||x||y|\|=\|x\|\|y\|.$
\end{itemize}
\end{proposition}

\begin{proof}
Let us firstly note that
\begin{equation}\label{eq1}
\begin{split}
\|x+y\|^2& =\||x|^2+2\Re(\scal{x,y})+|y|^2\|\\
& \le\||x|^2+|y|^2\|\le\|x\|^2+\|y\|^2.
\end{split}
\end{equation}
If $(i)$ holds true, then the inequalities in \eqref{eq1} become equalities. Also, by \cite[Proposition 3.3]{Kit}, the triangle equality $\||x|^2+|y|^2\|=\|x\|^2+\|y\|^2$ can be written in the form $\||x||y|\|=\|x\|\|y\|$, which is exactly the last condition of $(ii)$.
The converse follows the same path.
\end{proof}

\begin{theorem}\label{t6}
Let $x$ and $y$ be two elements in a pre-Hilbert $C^*$-module $\mathscr{E}$ such that $\Re(\scal{x,y})\le 0$. The following conditions are equivalent:
\begin{enumerate}[$(i)$]
\item $\|x+y\|^2=\|x\|^2+\|y\|^2$.
\item $\|x\|^2+\|y\|^2\in V(|x+y|^2)$.
\item There exists $\varphi\in S_{|x|^2}(\mathcal{A})\cap S_{|y|^2}(\mathcal{A})$ such that
$\varphi(\Re(\scal{x,y}))=0$.
\end{enumerate}
\end{theorem}

\begin{proof}
The implication $(i)\Rightarrow(ii)$ follows by \cite[Theorem 3.3.6]{Mur}.

$(ii)\Rightarrow(i)$. Conversely, if $\varphi$ is a state on $\mathcal{A}$ as in $(ii)$, then, by \eqref{eq1},
\begin{equation*}
\|x+y\|^2\le\|x\|^2+\|y\|^2=\varphi(|x+y|^2)\le\||x+y|^2\|=\|x+y\|^2.
\end{equation*}
Consequently, $(i)$ holds true.

$(iii)\Rightarrow(ii)$. Let $\varphi$ be a state on $\mathcal{A}$ such that
\begin{equation*}
\varphi(|x|^2)=\|x\|^2,\ \varphi(|y|^2)=\|y\|^2\text{ and }\varphi(\Re(\scal{x,y}))=0.
\end{equation*}
Then
\begin{equation*}
\varphi(|x+y|^2)=\varphi(|x|^2)+2\varphi(\Re(\scal{x,y}))+\varphi(|y|^2)=\|x\|^2+\|y\|^2.
\end{equation*}

$(ii)\Rightarrow(iii)$. Conversely, let $\varphi\in S(\mathcal{A})$ be a state which satisfies condition $(ii)$. Then
\begin{equation*}
\begin{split}
\|x\|^2+\|y\|^2 & =\varphi(|x|^2)+2\varphi(\Re(\scal{x,y}))+\varphi(|y|^2)\\
& \le\|x\|^2+2\varphi(\Re(\scal{x,y}))+\varphi(|y|^2)\\
& \le\|x\|^2+\varphi(|y|^2)\\
&\le\|x\|^2+\|y\|^2.
\end{split}
\end{equation*}
Thus $\varphi(|x|^2)=\|x\|^2$, $\varphi(|y|^2)=\|y\|^2$ and
$\varphi(\Re(\scal{x,y}))=0$. The statement $(iii)$ is proved.
\end{proof}

One can specialize this result for elements in the Hilbert $\mathcal{L}(\mathscr{H})$-module $\mathscr{E}=\mathcal{L}(\mathscr{H})$ (i.e., for bounded linear operators on $\mathscr{H}$). It is noted that for a bouned linear operator $A$ acting on a Hilbert space $\mathscr{H})$, the numerical range $V(A)$ is the closure of its classical numerical range $W(A):=\{\langle Ax,x\rangle: x\in\mathscr{H}, \|x\|=1\}$; see \cite{STA}.
\begin{corollary}
Let $A$ and $B$ be bounded linear operators acting on $\mathscr{H}$ such that $\Re(A^*B)\le 0$. The following conditions are equivalent:
\begin{itemize}
\item[$(i)$] $\|A+B\|^2=\|A\|^2+\|B\|^2$.
\item[$(ii)$] There exists a sequence $(\xi_n)_{n\ge 0}$ of unit vectors in $\mathscr{H}$ such that
\begin{equation*}
\|(A+B)\xi_n\|^2\xrightarrow{n\to\infty}\|A\|^2+\|B\|^2.
\end{equation*}
\item[$(iii)$] There exists a sequence $(\xi_n)_{n\ge 0}$ of unit vectors in $\mathscr{H}$ such that
\begin{equation*}
\|A\xi_n\|\xrightarrow{n\to\infty}\|A\|,\ \|B\xi_n\|\xrightarrow{n\to\infty}\|B\|\text{ and }\Re\scal{A\xi_n,B\xi_n}\xrightarrow{n\to\infty}0.
\end{equation*}
\end{itemize}
\end{corollary}

\begin{corollary}
Let $x$ and $y$ be two elements in a pre-Hilbert $C^*$-module $\mathscr{E}$ such that $\Re(\scal{x,y})\le 0$. The following conditions are equivalent:
\begin{itemize}
\item[$(i)$] $\|x+y\|^2=\|x\|^2+\|y\|^2.$
\item[$(ii)$] $\|\alpha x+\beta y\|^2=|\alpha|^2\|x\|^2+|\beta|^2\|y\|^2$ for certain (equivalently, for every) $\alpha,\beta\in\mathbb{C}$ with $\bar\alpha\beta>0$.
\end{itemize}
\end{corollary}

\begin{remark}
\textit{If $\Re(\scal{x,y})\le 0$, then, by the Pythagoras identity $\|x+y\|^2=\|x\|^2+\|y\|^2$, one can also obtain the following inequality:
\begin{equation*}
\|\alpha x+\beta y\|^2\ge|\alpha|^2\|x\|^2+|\beta|^2\|y\|^2
\end{equation*}
for every complex numbers $\alpha,\beta$ such that $\bar\alpha\beta$ is real.}

Indeed, if $\varphi$ is a state on $\mathcal{A}$ satisfying condition $(iii)$ of Theorem \ref{t6}, then, for every $\alpha,\beta\in\mathbb{C}$ with $\bar\alpha\beta\in\R$, it holds
\begin{equation*}
\begin{split}
\|\alpha x+\beta y\|^2& \ge\varphi(|\alpha x+\beta y|^2)\\
& =|\alpha|^2\varphi(|x|^2)+2\alpha\bar{\beta}\varphi(\Re(\scal{x,y}))+|\beta|^2\varphi(|y|^2)\\
& =|\alpha|^2\|x\|^2+|\beta|^2\|y\|^2.
\end{split}
\end{equation*}
\qed
\end{remark}

Under the stronger assumption $\Re(\scal{x,y})=0$ the Pythagoras identities associated to the pairs $(x,y)$ and, respectively, $(\alpha x,\beta y)$ (for $\bar{\alpha}\beta\in\R^*$) are actually equivalent.
\begin{corollary}\label{c36}
Let $x$ and $y$ be two elements in a pre-Hilbert $C^*$-module $\mathscr{E}$ such that $\Re(\scal{x,y})=0$. The following statements are equivalent:
\begin{enumerate}[$(i)$]
\item $\|x+y\|^2=\|x\|^2+\|y\|^2$.
\item $\|\alpha x+\beta y\|^2=|\alpha|^2\|x\|^2+|\beta|^2\|y\|^2$ for certain (equivalently, for every) non-null complex numbers $\alpha$ and $\beta$ with $\bar\alpha\beta\in\R$.
\item $\||x||y|\|=\|x\|\|y\|$.
\item $S_{|x+y|^2}(\mathcal{A})=S_{|x|^2}(\mathcal{A})\cap S_{|y|^2}(\mathcal{A})$.
\end{enumerate}
\end{corollary}

\begin{proof}
Under the assumption $\Re(\scal{x,y})=0$ condition $(i)$ takes the form $\||x|^2+|y|^2\|=\|x\|^2+\|y\|^2$. The equivalences between $(i),\ (ii),$ and $(iii)$ follow by Proposition \ref{c5}. By the same corollary the statements are also equivalent with $S_{|x|^2}(\mathcal{A})\cap S_{|y|^2}(\mathcal{A})\ne\emptyset$. So $(iv)$ implies $(i)$. Finally, if
$\|x+y\|^2=\|x\|^2+\|y\|^2$ (condition $(i)$ holds true), then, for any $\varphi\in S_{|x+y|^2}(\mathcal{A})$,
\begin{equation*}
\begin{split}
\|x+y\|^2& =\varphi(|x+y|^2)\\
& =\varphi(|x|^2)+\varphi(|y|^2)\\
& \le\|x\|^2+\|y\|^2.
\end{split}
\end{equation*}
Hence, $\varphi\in S_{|x|^2}(\mathcal{A})\cap S_{|y|^2}(\mathcal{A})$. In other words,
$S_{|x+y|^2}(\mathcal{A})\subseteq S_{|x|^2}(\mathcal{A})\cap S_{|y|^2}(\mathcal{A})$. The converse inclusion is obvious, so the statement $(iv)$ is verified.
\end{proof}

\begin{corollary}\label{c16}
Let $x$ and $y$ be two elements in a pre-Hilbert $C^*$-module $\mathscr{E}$ such that $\scal{x,y}=0$. The following statements are equivalent:
\begin{enumerate}[$(i)$]
\item $\|x+y\|^2=\|x\|^2+\|y\|^2$.
\item $\|\alpha x+\beta y\|^2=|\alpha|^2\|x\|^2+|\beta|^2\|y\|^2$ for certain (equivalently, for every) nonzero complex numbers $\alpha$ and $\beta$.
\end{enumerate}
\end{corollary}

\section{Pythagoras Orthogonality}

It is our aim in this section to investigate the Pythagoras orthogonality in the context of Hilbert $C^*$-modules. We list some properties of this notion, as follows:
\begin{enumerate}[$(a)$]
\item If $x$ and $y$ are linearly dependent, then $x\perp_P y$ if and only if $x=0$ or $y=0$. Due to this simple remark one may suppose, when trying to describe the concept of Pythagoras orthogonality, that the two vectors $x$ and $y$ are linearly independent. If not stated otherwise, we will make this assumption for the rest of the paper.
\item In inner product spaces, $x\perp_P y$ if and only if $\scal{x,y}=0$.
\item In pre-Hilbert $C^*$-modules, if $\scal{x,y}=0$, then $x\perp_P y$ if and only if $\|x+\alpha y\|^2=\|x\|^2+|\alpha|^2\|y\|^2$ for a certain non-null $\alpha\in\mathbb{C}$ if and only if $\||x||y|\|=\|x\|\|y\|$ (see Corollaries \ref{c36} and \ref{c16}).
\item In normed $*$-algebras, $x\perp_P y$ if and only if $x^*\perp_P y^*$.
\item $x\perp_P x$ if and only if $x=0$ (nondegenerate).
\item $x\perp_P y$ if and only if $y\perp_P x$ (symmetric).
\item If $x\perp_P y$, then $(\alpha x)\perp_P(\beta y),\ \alpha,\beta\in\mathbb{C}$ (homogeneous).
\item If $x\perp_P y$, then $x\perp_R y$, $x\perp_B y$ and $y\perp_B x$.
\item If $x\perp_P y$,
then $x$ and $y$ satisfy the parallelogram law.
\item If $x\perp_R y$, then $x\perp_P y$ if and only if $x$ and $y$ satisfy the parallelogram law.
\end{enumerate}

We describe, in a few examples, for elements $x,y$ in a pre-Hilbert $C^*$-module, the relationship between the Birkhoff--James, Roberts and Pythagoras orthogonality, the parallelogram law and the equality $\scal{x,y}=0$.

\begin{example}
Let us consider the $C^*$-algebra $\mathcal{L}(\mathscr{H})$ of bounded linear operators on a separable Hilbert space $\mathscr{H}$ ($\mathcal{L}(\mathscr{H})$ is regarded as a Hilbert module over itself). For a given orthonormal basis $(e_n)_{n\ge 1}$ in $\mathscr{H}$ we define $A,B\in\mathcal{L}(\mathscr{H})$ by
\begin{equation*}
Ae_n=\begin{cases}
\frac{1}{2}e_1& \text{if $n=1$},\\
\frac{1}{\sqrt{2^k}}e_2& \text{if $n=2k,\ k\ge 1$},\\
0& \text{otherwise}
\end{cases}
\text{ and }
Be_n=\begin{cases}
\frac{1}{2}e_1& \text{if $n=1$},\\
\frac{1}{\sqrt{2^k}}e_2& \text{if $n=2k+1,\ k\ge 1$},\\
0& \text{otherwise.}
\end{cases}
\end{equation*}
Then $\|A\|=\|B\|=1$ and, for any $\lambda\in \mathbb{C}$, we have
\begin{equation*}
\|A+\lambda B\|^2=1+|\lambda|^2=\|A\|^2+|\lambda|^2\|B\|^2.
\end{equation*}
Hence, \textit{$A$ and $B$ are orthogonal in the Pythagoras sense}. Moreover, since $e_1\in\ran A\cap\ran B$, $\scal{A,B}\neq 0$ ($\ran A$ denotes the range of $A$).\qed
\end{example}

\begin{example} Suppose that $\mathcal{A}$ is the $C^*$-algebra $\mathcal{C}[0,1]$ of all complex valued continuous functions on the closed interval $[0,1]$ (considered as a Hilbert $C^*$-module over itself). Let $f,g\in\mathcal{A}$ be defined by
\begin{equation*}
f(x)=\begin{cases}
\frac{1}{2}-x& \text{if $0\leq x\leq \frac{1}{2}$},\\
0& \text{if $\frac{1}{2}<x\leq 1$}
\end{cases}\text{ and }
g(x)=\begin{cases}
0& \text{if $0\leq x\leq \frac{1}{2}$},\\
x-\frac{1}{2}& \text{if $\frac{1}{2}< x\leq 1$}.
\end{cases}
\end{equation*}
Then $\|f\|=\|g\|=\frac{1}{2}$ and, for $\lambda\in\mathbb{C}$,
\begin{equation*}
\|f+\lambda g\|=\max\biggl\{\frac{1}{2},\frac{|\lambda|}{2}\biggr\}.
\end{equation*}
It follows that $f\perp_B g,\ g\perp_B f$ and $f\perp_R g$. Although $\scal{f,g}=0$, \textit{$f$ and $g$ are not orthogonal in the Pythagoras sense}, since $\|f+\lambda g\|^2=\|f\|^2+|\lambda|^2\|g\|^2$ if and only if $\lambda=0$. Equivalently, as $f\perp_R g$, \textit{$f$ and $g$ do not satisfy the parallelogram law}, either.\qed
\end{example}

\begin{example}\label{e2}
Let $S$ and $T$ be bounded linear operators on $\mathscr{H}$. If
\begin{equation*}
A=\begin{pmatrix} S&0\\0&0\end{pmatrix}\text{ and }B=\begin{pmatrix} 0&T\\0&0\end{pmatrix}
\end{equation*}
are elements of the $C^*$-algebra $\mathcal{A}=\mathcal{L}(\mathscr{H}\oplus\mathscr{H})$, then $\|A\|=\|S\|$, $\|B\|=\|T\|$ and
\begin{equation*}
\|A+\lambda B\|^2=\|SS^*+|\lambda|^2TT^*\|,\quad\lambda\in\mathbb{C}.
\end{equation*}
One can immediately verify that $A\perp_B B,\ B\perp_B A$ and $A\perp_R B$. Also, by Corollary \ref{c6}, \textit{$A$ and $B$ are orthogonal in the Pythagoras sense if and only if $\|S^*T\|=\|S\|\|T\|$}. In particular, this orthogonality condition is satisfied, for example, when $S$ is a scalar multiple of a coisometric operator or of an orthogonal projection $P$ with $\ran(P)\supseteq\ran(T)$.

In addition, one can immediately verify that, \textit{$A$ and $B$ satisfy the parallelogram law if and only if they are orthogonal in the Pythagoras sense}. By the other hand, \textit{$\scal{A,B}=A^*B=0$ if and only if $\ran S\perp\ran T$}.\qed
\end{example}

\begin{example}\label{e3}
For given $X\in\mathcal{L}(\mathscr{H})\setminus\{0\}$ and complex numbers $a,b,c,d$ we define (on the Hilbert space $\mathscr{H}\oplus\mathscr{H}$) the matrix operator
\begin{equation*}
M_X(a,b,c,d)=\begin{pmatrix} aI&bX\\cX^*&dI\end{pmatrix}.
\end{equation*}
It's norm has been computed by Feldman, Krupnik, and Markus in \cite[Lemma 1.6]{FKM} as
\begin{equation*}
\|M_X(a,b,c,d)\|=\frac{\sqrt{r-s}+\sqrt{r+s}}{2},
\end{equation*}
where $r=|a|^2+|d|^2+(|b|^2+|c|^2)\|X\|^2$ and $s=2|ad-bc\|X\||$.

Let $A=M_X(a,0,0,d)$ and $B=M_X(0,b,c,0)$. After direct calculations one may check that, regardless of the values of $a,b,c$ and $d$, $A\perp_B B$ and $A\perp_R B$. By the other hand, $B\perp_B A$ if and only if $bc=0$. We can also verify that \textit{the following conditions are equivalent:
\begin{itemize}
\item[$(i)$] $A$ and $B$ are orthogonal in the Pythagoras sense$;$
\item[$(ii)$] $A$ and $B$ satisfy the parallelogram law$;$
\item[$(iii)$] $ad=bc=0$.
\end{itemize}}
Finally, \textit{$\scal{A,B}=0$ if and only if $ab=cd=0$}.
\qed
\end{example}

\begin{example}\label{e4}
We study the concepts of orthogonality presented above for rank one operators, i.e., for operators of the form
\begin{equation*}
\mathscr{H}\ni z\mapsto x\otimes y(z):=\scal{z,y}x\in\mathscr{H},
\end{equation*}
where $x$ and $y$ are given vectors in $\mathscr{H}$.

Let $x,y,u,v\in\mathscr{H}$, and consider the operators $A=x\otimes y$ and $B=u\otimes v$. Then $\|A\|=\|x\|\|y\|$ and $\|B\|=\|u\|\|v\|$. The squared singular values of $A+\lambda B$ ($\lambda\in\mathbb{C}$) are the roots of the equation (in the unknown $t$)
\begin{multline*}
(t-\|x\|^2\|y\|^2-\bar{\lambda}\scal{x,u}\scal{v,y})(t-|\lambda|^2\|u\|^2\|v\|^2-\lambda\scal{u,x}\scal{y,v})\\
=(\lambda\scal{u,x}\|y\|^2+|\lambda|^2\scal{v,y}\|u\|^2)(\bar{\lambda}\scal{x,u}\|v\|^2+\scal{y,v}\|x\|^2).
\end{multline*}
Hence,
\begin{equation*}
\begin{split}
\|A+\lambda B\|^2& =\frac{1}{2}\Biggl(\|x\|^2\|y\|^2+|\lambda|^2\|u\|^2\|v\|^2+\lambda\scal{u,x}\scal{y,v}+\bar{\lambda}\scal{x,u}\scal{v,y}\\
& \qquad +\sqrt{
\begin{aligned}
(&\|x\|^2\|y\|^2+|\lambda|^2\|u\|^2\|v\|^2+\lambda\scal{u,x}\scal{y,v}+\bar{\lambda}\scal{x,u}\scal{v,y})^2\\
&\ -4|\lambda|^2\|x\|^2\|y\|^2\|u\|^2\|v\|^2-4|\lambda|^2|\scal{u,x}\scal{y,v}|^2\\
&\ +4|\lambda|^2\|x\|^2\|u\|^2|\scal{y,v}|^2+4|\lambda|^2\|y\|^2\|v\|^2|\scal{u,x}|^2
\end{aligned}
}
\Biggr).
\end{split}
\end{equation*}
After some computations one can observe that the following conditions are equivalent:
\begin{enumerate}[$(i)$]
\item $A\perp_B B$.
\item $B\perp_B A$.
\item $A\perp_R B$.
\item $\scal{x,u}=0$ or $\scal{y,v}=0$.
\end{enumerate}

Also, \textit{$A$ and $B$ are orthogonal in the Pythagoras sense if and only if one of the following conditions holds true:
\begin{itemize}
\item[$(a)$] $\{x,u\}$ are linearly dependent and $\scal{y,v}=0.$
\item[$(b)$] $\{y,v\}$ are linearly dependent and $\scal{x,u}=0$.
\end{itemize}}

Moreover, \textit{$A$ and $B$ satisfy the parallelogram law if and only if $\{x,u\}$ or $\{y,v\}$ are linearly dependent}. Finally, \textit{$\scal{A,B}=0$ if and only if $\scal{x,u}=0$}.\qed
\end{example}

Let $\mathcal{A}$ be a $C^*$-algebra. For an element $a\in\mathcal{A}$ we denote by $m(a)$ the \textit{minimum modulus} of $|a|$, i.e.,
\begin{equation*}
m(a):=\inf\{\varphi(|a|): \varphi\in S(\mathcal{A})\}.
\end{equation*}
Evidently, for a positive element $a\in\mathcal{A}$, $m(a)>0$ if and only if $a$ is invertible.
The following theorem, which relates Birkhoff--James orthogonality with a weaker version of Pythagoras orthogonality, has been formulated in the context of Hilbert spaces by Barraa and Boumazgour in \cite[Theorem 3]{BB} (see also \cite[Theorem 2.9]{ZM}).
\begin{theorem}\label{t10}
Let $x$ and $y$ be elements of a pre-Hilbert $\mathcal{A}$-module. The following conditions are equivalent:
\begin{enumerate}[$(i)$]
\item There exists $\varphi\in S_{|x|^2}(\mathcal{A})$ such that $\varphi(\scal{x,y})=0$.
\item $x\perp_B y$.
\item $\|x+\lambda y\|^2\ge\|x\|^2+|\lambda|^2m(|y|^2),\ \lambda\in\mathbb{C}$.
\end{enumerate}
\end{theorem}

\begin{proof}
The equivalence $(i)\Leftrightarrow(ii)$ has been obtained by Aramba\v{s}i\'c and Raji\'c in \cite[Theorem 2.7]{AR1} (see also \cite[Theorem 4.4]{BG}).

Clearly, $(ii)$ is a consequence of $(iii)$. Conversely, if $\varphi$ is a state of $\mathcal{A}$ which verifies $\varphi(|x|^2)=\|x\|^2$ and $\varphi(\scal{x,y})=0$ (by $(i)$), then, for every $\lambda\in\mathbb{C}$,
\begin{equation*}
\begin{split}
\|x+\lambda y\|^2& \ge\varphi(|x+\lambda y|^2)\\
& =\varphi(|x|^2)+\bar{\lambda}\varphi(\scal{x,y})+\lambda\varphi(\scal{y,x})+|\lambda|^2\varphi({|y|^2})\\
& =\|x\|^2+|\lambda|^2\varphi({|y|^2})\\
& \ge\|x\|^2+|\lambda|^2m(|y|^2).
\end{split}
\end{equation*}
Condition $(iii)$ is proved.
\end{proof}

It was noted by James \cite[Corollary 2.2]{Jam} that, for any two elements $x$ and $y$ of a normed linear space $\mathscr{X}$ there exists a number $\alpha$ such that $y\perp_B(x+\alpha y)$. Such a condition is not valid, in full generality, for Pythagoras orthogonality. However, a weaker version might still be formulated. Its operator version can be found in \cite[Corollary 4]{BB} (see also \cite[Corollary 2.11]{ZM}).
\begin{corollary}
Let $x$ and $y$ be elements of a pre-Hilbert module over $\mathcal{A}$ such that $m(|y|^2)>0$. Then there exists a unique $\alpha_0\in\mathbb{C}$ such that
\begin{equation}\label{eq2}
\|x+\alpha_0y+\lambda y\|^2\ge\|x+\alpha_0y\|^2+|\lambda|^2m(|y|^2),\quad\lambda\in\mathbb{C}.
\end{equation}
Moreover, one can find $\varphi\in S_{|x+\alpha y|^2}(\mathcal{A})$ such that $\varphi(\scal{x+\alpha y,y})=0$ if and only if $\alpha=\alpha_0$.
\end{corollary}

\begin{proof}
We firstly observe that, since $\lim_{|\alpha|\to\infty}\|x+\alpha y\|=\infty$,
\begin{equation*}
\inf\{\|x+\alpha y\|:\alpha\in\mathbb{C}\}=\inf\{\|x+\alpha y\|:|\alpha|\le\delta\}
\end{equation*}
for a certain $\delta>0$. In addition, as the map $\alpha\mapsto\|x+\alpha y\|$ is continuous on the compact set $\{|\alpha|\le\delta\}$, it attains its minimum at some point $\alpha_0\in\mathbb{C}$, that is, $\|x+\alpha_0y+\lambda y\|\ge\|x+\alpha_0y\|$ for every $\lambda\in\mathbb{C}$. Formula \eqref{eq2} then follows by Theorem \ref{t10} $(iii)$. If, for some $\alpha_1\in \mathbb{C}$,
\begin{equation*}
\|x+\alpha_1y+\lambda y\|^2\ge\|x+\alpha_1y\|^2+|\lambda|^2m(|y|^2),\quad\lambda\in\mathbb{C}
\end{equation*}
then, by taking $\lambda=\alpha_0-\alpha_1$, we obtain
\begin{equation*}
m(|y|^2)|\alpha_1-\alpha_0|^2\le\|x+\alpha_0y\|^2-\|x+\alpha_1y\|^2\le 0,
\end{equation*}
so, $\alpha_1=\alpha_0$. The final statement is a consequence, in view of the uniqueness of $\alpha_0$, of Theorem \ref{t10} ($(i)\Leftrightarrow(iii)$).
\end{proof}

As seen earlier, the Pythagoras orthogonality implies both the parallelogram law and Birkhoff--James orthogonality. In certain particular situations, the converse is also true.
\begin{theorem}
Let $\mathcal{A}$ be a unital $C^*$-algebra with unit $e$. If $x$ and $y$ are elements in a pre-Hilbert $\mathcal{A}$-module such that $|y|^2=\alpha e$ ($\alpha>0$ is given), then the following conditions are equivalent:
\begin{enumerate}[$(i)$]
\item $x$ and $y$ are orthogonal in the Pythagoras sense.
\item $x$ and $y$ satisfy the parallelogram law and are Birkhoff--James orthogonal.
\end{enumerate}
\end{theorem}

\begin{proof}
The direct implication is obvious. Conversely, if $(ii)$ holds true, then, by Theorem \ref{t10} $(iii)$,
\begin{equation*}
\begin{split}
\|x+\lambda y\|^2& \ge\|x\|^2+|\lambda|^2m(|y|^2)\\
& =\|x\|^2+\alpha|\lambda|^2\\
& =\|x\|^2+|\lambda|^2\|y\|^2,
\quad\lambda\in\mathbb{C}.
\end{split}
\end{equation*}
In view of the parallelogram law, the inequalities above become equalities. Hence $x\perp_P y$, as required.
\end{proof}

The operator version of Theorem \ref{t10} ($(i)\Leftrightarrow(ii)$) has been obtained by B. Magajna in \cite[Lemma 2.2]{Mag} (see also \cite[Remark 3.1]{BS}). It states that two bounded linear operators $A$ and $B$ on $\mathscr{H}$ are orthogonal in the Birkhoff--James sense
if and only if there exists a sequence $(\xi_n)_{n\ge 0}$ of unit vectors in $\mathscr{H}$ such that
\begin{equation*}
\|A\xi_n\|\xrightarrow{n\to\infty}\|A\|\text{ and }\scal{A\xi_n,B\xi_n}\xrightarrow{n\to\infty}0.
\end{equation*}

For Pythagoras orthogonality (a concept which is stronger than Birkhoff--James orthogonality) we must include certain additional conditions. One of the main tools in our developments is the following expression of the $\min_{\lambda\in\mathbb{C}}\|A+\lambda B\|$.
\begin{theorem}[{\cite[Proposition 2.1]{AR1}}]\label{Thraj}
Let $A$ and $B$ be bounded linear operators on $\mathscr{H}$. Then
\begin{equation*}
\min_{\lambda\in\mathbb{C}}\|A+\lambda B\|^2=\sup_{\|\xi\|=1}M_{A,B}(\xi),
\end{equation*}
where
\begin{equation*}
M_{A,B}(\xi)=\begin{cases}
\|A\xi\|^2-\frac{|\langle A\xi, B\xi\rangle|^2}{\|B\xi\|^2}& \text{if $B\xi\neq 0$},\\
\|A\xi\|^2& \text{if $B\xi=0$}.
\end{cases}
\end{equation*}
\end{theorem}

The following limit characterization provides a step forward in our desired description of Pythagoras orthogonality.
\begin{proposition}\label{p13}
Let $A$ and $B$ be bounded linear operators acting on $\mathscr{H}$ such that
\begin{equation*}
\|(1+\lambda_0)A+\lambda_0\alpha B\|^2=(1+\lambda_0)^2\|A\|^2+\lambda_0^2|\alpha|^2\|B\|^2
\end{equation*}
for a certain $\lambda_0\in\R\setminus\{-1,0\}, \alpha\in \mathbb{C}\setminus\{0\}$ and $(\xi_n)_{n\ge 0}$ a sequence of unit vectors in $\mathscr{H}$.
\begin{enumerate}[$(a)$]
\item If $(A+\alpha B)\xi_n\ne 0,\ n\ge 0$ and
\begin{equation}\label{eqM}
M_{A,A+\alpha B}(\xi_n)\xrightarrow{n\to\infty}(1+\lambda_0)^2\|A\|^2+\lambda_0^2|\alpha|^2\|B\|^2,
\end{equation}
and  $(x_n)$ is a subsequence of $(\xi_n)$ such that the limits  
\begin{equation}\label{eqlim}
a=\lim_{n\to\infty}\|Ax_n\|,\ b=\lim_{n\to\infty}\|Bx_n\|\text{ and }c=\lim_{n\to\infty}\scal{Ax_n,Bx_n}
\end{equation}
exist, then $a, b$, and $c$ satisfy the conditions
\begin{equation}\label{eqabc}
a^2(\lambda_0+1)+c\bar{\alpha}\lambda_0=-b^2|\alpha|^2\lambda_0-c\bar{\alpha}(\lambda_0+1)=(1+\lambda_0)^2\|A\|^2+\lambda_0^2|\alpha|^2\|B\|^2.
\end{equation}
Moreover,
\begin{equation*}
\|(1+\lambda_0)A+\lambda_0\alpha B\|=\min_{\lambda\in\mathbb{C}}\|(1+\lambda)A+\lambda\alpha B\|
\end{equation*}
and
\begin{equation}\label{eq6p13}
\|A+\lambda B\|^2\ge\dfrac{\splitdfrac{[(1+\lambda_0)^2\|A\|^2+\lambda_0^2|\alpha|^2\|B\|^2][\lambda_0|\alpha|^2-(\lambda_0+1)|\lambda|^2]}{-\bar{\alpha}c|\lambda_0\alpha-(\lambda_0+1)\lambda|^2}}{|\alpha|^2\lambda_0(\lambda_0+1)},\quad\lambda\in\mathbb{C}.
\end{equation}
\item Conversely, if the limits \eqref{eqlim} exist, satisfy conditions \eqref{eqabc} and $a^2\ne(1+\lambda_0)^2\|A\|^2+\lambda_0^2|\alpha|^2\|B\|^2$, then \eqref{eqM} holds true.
\end{enumerate}
\end{proposition}

\begin{proof}
We may assume, without loss of generality, that $\alpha=1$ ($B$ can be replaced by $\frac{1}{\alpha}B$, if necessary).

Let us now observe that, for every $n\ge 0$, the following inequalities hold true:
\begin{equation*}
\begin{split}
M_{A,A+B}(\xi_n)& =\|[(1+\lambda_0)A+\lambda_0B]\xi_n\|^2-\frac{|\scal{[(1+\lambda_0)A+\lambda_0B]\xi_n,(A+B)\xi_n}|^2}{\|(A+B)\xi_n\|^2}\\
& \le\|[(1+\lambda_0)A+\lambda_0B]\xi_n\|^2-\frac{|\scal{[(1+\lambda_0)A+\lambda_0B]\xi_n,(A+B)\xi_n}|^2}{\|A+B\|^2}\\
& \le\|[(1+\lambda_0)A+\lambda_0B]\xi_n\|^2\\
& \le\|[(1+\lambda_0)A+\lambda_0B]\|^2=(1+\lambda_0)^2\|A\|^2+\lambda_0^2\|B\|^2.
\end{split}
\end{equation*}
Letting $n\to\infty$ we conclude that \eqref{eqM} is equivalent with the following limit conditions:
\begin{gather}
\frac{\scal{[(1+\lambda_0)A+\lambda_0B]\xi_n,(A+B)\xi_n}}{\|(A+B)\xi_n\|}\xrightarrow{n\to\infty}0,\label{eq1p}\\
\scal{[(1+\lambda_0)A+\lambda_0B]\xi_n,(A+B)\xi_n}\xrightarrow{n\to\infty}0\label{eq1s}
\end{gather}
and
\begin{equation}\label{eq1t}
\|[(1+\lambda_0)A+\lambda_0B]\xi_n\|^2\xrightarrow{n\to\infty}(1+\lambda_0)^2\|A\|^2+\lambda_0^2\|B\|^2.
\end{equation}
Following the notations of \eqref{eqlim} one can write \eqref{eq1s} as
\begin{equation*}
(1+\lambda_0)a^2+\lambda_0b^2+(1+\lambda_0)c+\lambda_0\bar{c}=0.
\end{equation*}
Similarly, \eqref{eq1t} takes the form
\begin{equation*}
(1+\lambda_0)^2a^2+\lambda_0^2b^2+2\lambda_0(\lambda_0+1)\Re c=(1+\lambda_0)^2\|A\|^2+\lambda_0^2\|B\|^2.
\end{equation*}
Easy computations then show that \eqref{eq1s} and \eqref{eq1t} are actually equivalent with \eqref{eqabc}.

According to these remarks, in order to prove $(a)$, it only remains to let $n\to\infty$ into the formulas
\begin{equation*}
M_{A,A+B}(\xi_n)\le\min_{\lambda\in\mathbb{C}}\|(1+\lambda)A+\lambda B\|^2\le\|(1+\lambda_0)A+\lambda_0B\|^2,\quad n\ge 0
\end{equation*}
and
\begin{equation}\label{eqalb}
\begin{split}
\|A+\lambda B\|^2& \ge\|(A+\lambda B)\xi_n\|^2\\
& =\|A\xi_n\|^2+2\Re\lambda\scal{A\xi_n,B\xi_n}+|\lambda|^2\|B\xi_n\|^2,\quad\lambda\in\mathbb{C},\ n\ge 0.
\end{split}
\end{equation}

$(b)$ As seen above, \eqref{eq1s} and \eqref{eq1t} are a consequence of \eqref{eqabc}. In addition,
\begin{equation*}
\|(A+B)\xi_n\|^2\xrightarrow{n\to\infty}\dfrac{a^2-[(1+\lambda_0)^2\|A\|^2+\lambda_0^2\|B\|^2]}{\lambda_0^2}.
\end{equation*}
Hence, $\lim_{n\to\infty}\|(A+B)\xi_n\|>0$, which shows that \eqref{eq1p} also holds true. The proof is completed.
\end{proof}

\begin{lemma}\label{l15}
Let $A$ and $B$ be bounded linear operators on $\mathscr{H}$ such that $\rank(A+\alpha_i B)=1$ for pairwise distinct complex numbers $\alpha_i\in\mathbb{C},\ i=1,2,3$. Then $\rank(A+\alpha B)=1$ for every $\alpha\in\mathbb{C}$.
\end{lemma}

\begin{proof}
Let $x_1,x_2,y_1,y_2$ be non-null vectors in $\mathscr{H}$ such that $A+\alpha_1 B=x_1\otimes y_1$ and $A+\alpha_2 B=x_2\otimes y_2$. Then
\begin{equation*}
A+\lambda B=\frac{\lambda-\alpha_2}{\alpha_1-\alpha_2}x_1\otimes y_1+\frac{\alpha_1-\lambda}{\alpha_1-\alpha_2}x_2\otimes y_2,\quad\lambda\in\mathbb{C}.
\end{equation*}

We distinguish two cases:

$(1)$ $\{x_1,x_2\}$ are linearly independent. Since $\rank(A+\alpha_3 B)=1$, one can find $\beta_1,\beta_2\in\mathbb{C}$ (at least one of them is non-null) such that $\ran(A+\alpha_3 B)=\spc(\beta_1x_1+\beta_2x_2)$. Then, for every $z\in\mathscr{H}$, there exists $\mu\in\mathbb{C}$ such that
\begin{equation*}
\frac{\alpha_3-\alpha_2}{\alpha_1-\alpha_2}\scal{z,y_1}=\mu\beta_1\text{ and }\frac{\alpha_1-\alpha_3}{\alpha_1-\alpha_2}\scal{z,y_2}=\mu\beta_2.
\end{equation*}
Therefore, $\beta_1,\beta_2$ are both non-null and $\frac{\bar{\beta_2}(\bar{\alpha_3}-\bar{\alpha_2})}{\bar{\alpha_1}-\bar{\alpha_2}}y_1-\frac{\bar{\beta_1}(\bar{\alpha_1}-\bar{\alpha_3})}{\bar{\alpha_1}-\bar{\alpha_2}}y_2=0$, so $\{y_1,y_2\}$ are linearly dependent. In other words,
\begin{equation*}
A+\lambda B=\biggl[\frac{\lambda-\alpha_2}{\alpha_1-\alpha_2}x_1+\frac{\beta_2(\alpha_3-\alpha_2)(\alpha_1-\lambda)}{\beta_1(\alpha_1-\alpha_2)(\alpha_1-\alpha_3)}x_2\biggr]\otimes y_1,\quad\lambda\in\mathbb{C}.
\end{equation*}

$(2)$ $\{x_1,x_2\}$ are linearly dependent. In this case there exists a complex number $\beta\ne 0$ such that $x_2=\beta x_1$. We conclude that
\begin{equation*}
A+\lambda B=x_1\otimes\biggl[\frac{\bar{\lambda}-\bar{\alpha_2}}{\bar{\alpha_1}-\bar{\alpha_2}}y_1+\frac{\bar{\beta}(\bar{\alpha_1}-\bar{\lambda})}{\bar{\alpha_1}-\bar{\alpha_2}}y_2\biggr],\quad\lambda\in\mathbb{C}.
\end{equation*}
Since $\{A,B\}$ are linearly independent (as assumed earlier; this also implies that $\{y_1,y_2\}$ are linearly independent) we deduce that $A+\lambda B$ has rank one for every $\lambda\in\mathbb{C}$.
\end{proof}

We are now ready to present the announced characterization of Pythagoras orthogonality.
\begin{theorem}\label{t16}
Let $A$ and $B$ be bounded linear operators acting on $\mathscr{H}$ such that $\rank(A+\alpha_1 B)>1$ and $\Re(\alpha_2A^*B)\ge 0$ for certain $\alpha_1,\alpha_2\in\mathbb{C},\ \alpha_2\ne 0$. The following conditions are equivalent:
\begin{itemize}
\item[$(i)$] $A$ and $B$ are orthogonal in the Pythagoras sense.
\item[$(ii)$] $A$ and $B$ verify the parallelogram law and there exists a sequence $(\xi_n)_{n\ge 0}$ of unit vectors in $\mathscr{H}$ such that
\begin{equation*}
\|A\xi_n\|\xrightarrow{n\to\infty}\|A\|,\ \|B\xi_n\|\xrightarrow{n\to\infty}\|B\|\text{ and }\scal{A\xi_n,B\xi_n}\xrightarrow{n\to\infty}0.
\end{equation*}
\item[$(iii)$] $A$ and $B$ verify the parallelogram law and there exists a sequence $(\xi_n)_{n\ge 0}$ of unit vectors in $\mathscr{H}$ such that
\begin{equation*}
\|(A+\lambda B)\xi_n\|^2\xrightarrow{n\to\infty}\|A\|^2+|\lambda|^2\|B\|^2\text{ for every }\lambda\in\mathbb{C}.
\end{equation*}
\end{itemize}
\end{theorem}

\begin{proof}
$(i)\Rightarrow(ii)$. The parallelogram law is obviously weaker than (or, at most equivalent to) Pythagoras orthogonality.

Our next aim is to prove the limit conditions of $(ii)$. Since $\Re(2\alpha_2A^*B)\ge 0$ and $\Re(3\alpha_2A^*B)\ge 0$ and, by Lemma \ref{l15}, at least one of the operators $A+\alpha_2 B$, $A+2\alpha_2 B$ and $A+3\alpha_2 B$ has rank strictly greater than one, so we can assume that $\alpha_1=\alpha_2=\alpha$. As we have previously done we can also assume that $\alpha=1$. We firstly observe that, by $(i)$,
\begin{equation*}
\|(1+\lambda)A+\lambda B\|^2=|1+\lambda|^2\|A\|^2+|\lambda|^2\|B\|^2,\quad\lambda\in\mathbb{C}.
\end{equation*}
An easy computation then shows that
\begin{equation}\label{eq3}
\min_{\lambda\in\mathbb{C}}\|(1+\lambda)A+\lambda B\|^2=\frac{\|A\|^2\|B\|^2}{\|A\|^2+\|B\|^2},
\end{equation}
which is attained for $\lambda_0=-\frac{\|A\|^2}{\|A\|^2+\|B\|^2}$.
It follows from Theorem \ref{Thraj} and \eqref{eq3} that there exists a sequence $(\xi_n)_{n\ge 0}$ of unit vectors in $\mathscr{H}$ such that
\begin{equation}\label{eq4}
M_{A,A+B}(\xi_n)\xrightarrow{n\to\infty}\frac{\|A\|^2\|B\|^2}{\|A\|^2+\|B\|^2}.
\end{equation}

We may suppose, eventually on a subsequence, that $(A+B)\xi_n\ne 0$ for every $n\ge 0$. Indeed, if, otherwise, $(A+B)\xi_n=0$ for every $n\ge n_0$ and for a certain $n_0\ge 0$, then \eqref{eq4} takes the form
\begin{equation}\label{eq5}
\|A\xi_n\|^2=\|B\xi_n\|^2\xrightarrow{n\to\infty}\frac{\|A\|^2\|B\|^2}{\|A\|^2+\|B\|^2}
\end{equation}
by the definition of $M_{A,A+B}$. As $(A\xi_n)_{n\ge 0}$ is a bounded sequence in $\mathscr{H}$ it contains a weakly convergent subsequence (denoted also by $(A\xi_n)_{n\ge 0}$) to a vector $w\in\mathscr{H}$. Obviously, $\spc\{(A+B)^*w\}\subsetneq\overline{\ran(A+B)^*}$ as the rank of $(A+B)^*$ is strictly greater than $1$. Consequently, one can find a unit vector $e\in\overline{\ran(A+B)^*}$, which is orthogonal to $(A+B)^*w$. Then, by setting $u_n=\sqrt{\frac{n}{n+1}}\xi_n+\frac{1}{\sqrt{n+1}}e,\ n\ge 0$, we have
\begin{equation*}
(A+B)u_n=\sqrt{\frac{n}{n+1}}(A+B)\xi_n+\frac{1}{\sqrt{n+1}}(A+B)e=\frac{1}{\sqrt{n+1}}(A+B)e\ne 0,\quad n\ge n_0,
\end{equation*}
since $e\perp\ker(A+B)$. Moreover, for $n\ge 0$,
\begin{equation}\label{eq6}
\begin{split}
M_{A,A+B}(u_n)& =\|Au_n\|^2-\frac{|\scal{Au_n,(A+B)u_n}|^2}{\|(A+B)u_n\|^2}\\
& =\frac{n}{n+1}\|A\xi_n\|^2+\frac{1}{n+1}\|Ae\|^2+2\frac{\sqrt{n}}{n+1}\Re\scal{A\xi_n,Ae}\\
& \qquad -\frac{\bigl|\bigl\langle\sqrt{\frac{n}{n+1}}A\xi_n+\frac{1}{\sqrt{n+1}}Ae,\frac{1}{\sqrt{n+1}}(A+B)e\bigr\rangle\bigr|^2}{\bigl\|\frac{1}{\sqrt{n+1}}(A+B)e\bigr\|^2}\\
& =\frac{n}{n+1}\|A\xi_n\|^2+\frac{1}{n+1}\|Ae\|^2+2\frac{\sqrt{n}}{n+1}\Re\scal{A\xi_n,Ae}\\
& \qquad -\frac{\bigl|\sqrt{\frac{n}{n+1}}\langle A\xi_n,(A+B)e\rangle+\frac{1}{\sqrt{n+1}}\bigl\langle Ae,(A+B)e\bigr\rangle\bigr|^2}{\|(A+B)e\|^2}.
\end{split}
\end{equation}
In view of \eqref{eq5} and the observation that
\begin{equation*}
\langle A\xi_n,(A+B)e\rangle\xrightarrow{n\to\infty}\scal{w,(A+B)e}=\scal{(A+B)^*w,e}=0,
\end{equation*}
we deduce, by passing to limit in \eqref{eq6}, that
\begin{equation*}
M_{A,A+B}(u_n)\xrightarrow{n\to\infty}\frac{\|A\|^2\|B\|^2}{\|A\|^2+\|B\|^2}.
\end{equation*}
One may consider, in this particular situation, the sequence $(u_{n+n_0})_{n\ge 0}$ which will be also denoted by $(\xi_n)_{n\ge 0}$.

The assumptions of Proposition \ref{p13} $(a)$ are verified. So, the limits \eqref{eqlim} satisfy the conditions (equivalent with \eqref{eqabc})
\begin{equation}\label{eq15}
c=\frac{\|A\|^2(b^2-\|B\|^2)}{\|B\|^2}=\frac{\|B\|^2(a^2-\|A\|^2)}{\|A\|^2}.
\end{equation}
We deduce that $c\le 0$. Also, by hypothesis (i.e., $\Re (A^*B)\ge 0$),
\begin{equation*}
c=\Re c=\lim_{n\to\infty}\Re\scal{\xi_n,A^*B\xi_n}=\lim_{n\to\infty}\scal{\xi_n,\Re(A^*B)\xi_n}\ge 0.
\end{equation*}
This forces $c=0$ and, by \eqref{eq15}, $a=\|A\|$ and $b=\|B\|$.

$(ii)\Rightarrow(iii)$. Clearly, if a sequence $(\xi_n)_{n\ge 0}$ of unit vectors in $\mathscr{H}$ verifies $(ii)$, then
\begin{multline*}
\|(A+\lambda B)\xi_n\|^2\\
=\|A\xi_n\|^2+2\Re\bar{\lambda}\scal{A\xi_n,B\xi_n}+|\lambda|^2\|B\xi_n\|^2\xrightarrow{n\to\infty}\|A\|^2+|\lambda|^2\|B\|^2,\quad\lambda\in\mathbb{C}.
\end{multline*}
Hence $(\xi_n)_{n\ge 0}$ also verifies $(iii)$.

$(iii)\Rightarrow(i)$. Let $(\xi_n)_{n\ge 0}$ be a sequence of unit vectors in $\mathscr{H}$ such that $(iii)$ holds true. Then
\begin{equation*}
\|A+\lambda B\|^2\ge\|(A+\lambda B)\xi_n\|^2\xrightarrow{n\to\infty}\|A\|^2+|\lambda|^2\|B\|^2.
\end{equation*}
The proof is finished, as before, by the use of the parallelogram law.
\end{proof}

\begin{remark}
$(a)$ \textit{The necessity of the rank condition}. Let $A$ and $B$ be bounded linear operators on $\mathscr{H}$ such that, for every $\alpha\in\mathbb{C}$, $A+\alpha B$ is a rank one operator. Then, as seen in the proof of Lemma \ref{l15}, one of the following two situations can occur:

$(1)$ $(A,B)=(x\otimes y_A,x\otimes y_B)$, with $x,y_A,y_B\in\mathscr{H}$, $x\ne 0$ and $\{y_A,y_B\}$ linearly independent. Then, by Example \ref{e4}, $(A,B)$ satisfies the parallelogram law. Moreover, $A\perp_P B$ if and only if $\scal{y_A,y_B}=0$. Condition $(ii)$ of Theorem \ref{t16} takes, for a given sequence $(\xi_n)_{n\ge 0}$ of unit vectors in $\mathscr{H}$, the form:
\begin{equation*}
|\scal{\xi_n,y_A}|\xrightarrow{n\to\infty}\|y_A\|,\ |\scal{\xi_n,y_B}|\xrightarrow{n\to\infty}\|y_B\|\text{ and }\scal{\xi_n,y_A}\overline{\scal{\xi_n,y_B}}\xrightarrow{n\to\infty}0.
\end{equation*}
This is, however, impossible.

$(2)$ $(A,B)=(x_A\otimes y,x_B\otimes y)$, with $x_A,x_B,y\in\mathscr{H}$, $\{x_A,x_B\}$ linearly independent and $y\ne 0$. Again by Example \ref{e4}, $(A,B)$ satisfies the parallelogram law. We also observe that a sequence $(\xi_n)_{n\ge 0}$ of unit vectors in $\mathscr{H}$ with $|\scal{\xi_n,y}|\xrightarrow{n\to\infty}\|y\|$ satisfies condition $(ii)$ of Theorem \ref{t16} if and only if $\scal{x_A,x_B}=0$ or, equivalently, $A\perp_P B$. By the other hand, $A\perp_P B$ if and only if $A^*\perp_P B^*$, but condition $(ii)$ of Theorem \ref{t16} is not verified for the pair $(A^*,B^*)$ (as seen in case $(1)$).

$(b)$ \textit{The necessity of the parallelogram law}. Let us consider, as in Example \ref{e3},
$A=M_I(1,0,0,1)$, $B=M_I(0,1,1,0)$ and $h\in\mathscr{H}$ a vector of unit norm. Then $A$ and $B$ do not satisfy the parallelogram law while, for $\xi=(h,0)$, $\|A\xi\|=\|A\|=1,\ \|B\xi\|=\|B\|=1$ and $\scal{A\xi,B\xi}=\scal{(h,0),(0,h)}=0$.

$(c)$ \textit{The importance of the condition $\Re(\alpha A^*B)\ge 0$ for certain nonzero $\alpha\in\mathbb{C}$}. Let $x$ and $y$ be unit vectors in $\mathscr{H}$ such that $\scal{x,y}=0$ (it is assumed that $\dim\mathscr{H}\ge 2$). If $S=x\otimes x+y\otimes y$ and $T=x\otimes y$, then the operators $A=\begin{pmatrix} S&0\\0&0\end{pmatrix}$ and $B=\begin{pmatrix} 0&T\\0&0\end{pmatrix}$ are orthogonal in the Pythagoras sense (according to Example \ref{e2}) and $\rank(A+\alpha B)=2$ for every $\alpha\in\mathbb{C}$. Moreover, for $h_1,h_2\in\mathscr{H}$ and $\alpha\in\mathbb{C},\ \alpha\ne 0$, it holds
\begin{equation*}
\scal{\Re(\alpha A^*B)(h_1,h_2),(h_1,h_2)}=\Re\scal{Sh_1,\alpha Th_2}=\Re(\bar{\alpha}\scal{h_1,x})\overline{\scal{h_2,y}}).
\end{equation*}
Hence,
\begin{equation*}
\scal{\Re(\alpha A^*B)(\alpha x,y),(\alpha x,y)}=-\scal{\Re(\alpha A^*B)(-\alpha x,y),(-\alpha x,y)}=|\alpha|^2>0.
\end{equation*}
Also, condition $(ii)$ of Theorem \ref{t16} can be expressed by the existence of sequences $(\xi_n)_{n\ge 0}$ and $(\eta_n)_{n\ge 0}$ of vectors in $\mathscr{H}$ with $\|\xi_n\|^2+\|\eta_n\|^2=1,\ n\ge 0$ such that
\begin{equation*}
|\scal{\xi_n,x}|^2+|\scal{\xi_n,y}|^2\xrightarrow{n\to\infty}1,\ |\scal{\eta_n,y}|\xrightarrow{n\to\infty}1\text{ and }\scal{\xi_n,x}\overline{\scal{\eta_n,y}}\xrightarrow{n\to\infty}0.
\end{equation*}
Equivalently,
\begin{equation*}
\scal{\xi_n,x}\xrightarrow{n\to\infty}0,\ |\scal{\xi_n,y}|\xrightarrow{n\to\infty}1\text{ and }|\scal{\eta_n,y}|\xrightarrow{n\to\infty}1.
\end{equation*}
Letting $n\to\infty$ in the Cauchy--Schwarz inequalities $|\scal{\xi_n,y}|\le\|\xi_n\|$ and $|\scal{\eta_n,y}|\le\|\eta_n\|$ ($n\ge 0$) we deduce that the limits $\lim_{n\to\infty}\|\xi_n\|$ and $\lim_{n\to\infty}\|\eta_n\|$ exist and they are both equal to $1$. This contradicts, however, the equality $\|\xi_n\|^2+\|\eta_n\|^2=1,\ n\ge 0$.\qed
\end{remark}

\end{document}